\documentclass[11pt]{article}
\usepackage{a4,amssymb,amsthm,amsmath,cite,comment}

\theoremstyle{plain}
\newtheorem{thm}{Theorem}
\newtheorem{prop}{Proposition}

\newtheorem{lem}{Lemma}

\theoremstyle{definition}

\newtheorem{prob}{Problem}
\newtheorem{ack}{Acknowledgements}

\nonstopmode

\def\cprime{$'$}
\newcommand{\nicefrac}[2]{\leavevmode\kern.1em
\raise.5ex\hbox{\the\scriptfont0 #1}\kern-.1em
/\kern-.15em\lower.25ex\hbox{\the\scriptfont0 #2}}
\def\halv{\mathchoice{{ \frac 1 2}}{1/2}{1/2}{1/2}}

\newcommand{\C}{{\mathbb C}} 
\newcommand{\R}{{\mathbb R}} 
\newcommand{\abs}[1]{{\left| {#1} \right|}} \newcommand{\p}[1]{{\left(
      {#1} \right)}}

\renewcommand{\Re}{\operatorname{Re}}

\renewcommand{\Im}{\operatorname{Im}}

 \newcommand{\Oh}[1]{{O \p{#1}}}

\begin{document}

\date{}

\author{Johan Andersson\thanks{Email:johan.andersson@his.se \, Address:School of Engineering science, University of Sk\"ovde, Box 408, 541 28 Sk\"ovde, SWEDEN}}

\title{Non universality on the critical line}
\maketitle

\begin{abstract}
  We prove that the Riemann zeta-function is not universal on the critical line  by using the fact that the Hardy $Z$-function is real, and some elementary considerations. This is a related to a recent result of Garunk{\v{s}}tis and Steuding. We also prove conditional and partial results for non universality on the lines $\Re(s)=\sigma$ for $0<\sigma<\halv$ and together with our recent result for non universality on the line $\Re(s)=1$ it will mostly answer the question of on what lines the zeta-function is universal. 
\end{abstract}

\tableofcontents

\section{Introduction}

Voronin \cite{Voronin2,Voronin0} proved the following Theorem:
\begin{thm}  Let 
 $D=\{ z \in \C :|z-3/4| \leq r<1 \}$, and suppose that $f$ is any continuous function that is nonvanishing and analytic on  $D$. Then 
\begin{gather*} 
 \liminf_{T \to \infty} \max_{z  \in D} \abs{\zeta(z+iT)-f(z)}=0.
\end{gather*}
\end{thm}
Bagchi \cite{Bagchi}  generalized this result to any compact set $D$ with connected complement lying entirely within $\halv<\Re(s)<1$. These results have resulted in massive international interest and an entire sub-field of analytic number theory concerned with``universality of $L$-functions''. For good introductions to this field, see the monographs  of Laurin{\v{c}}ikas \cite{Laurincikas}, or Steuding \cite{Steuding}.

 By choosing the special case $D=[\sigma,\sigma+iH]$ for $\halv<\sigma<1$ in Theorem 1 we obtain:
 \begin{thm}  Suppose $\halv<\sigma<1$ and that $f$ is any continuous function on the interval $[0,H]$.  Then 
\begin{gather*}\liminf_{T \to \infty} \max_{t \in [0,H]} \abs{\zeta(\sigma+it+iT)-f(t)}=0.\end{gather*} \end{thm}
The fact that we do not need to assume that $f$ is nonvanishing as in Theorem 1, is proved in \cite{Andersson2} whenever  $D$ is without interior points\footnote{For the case with interior points, partial results have been obtained in this direction. For some sets $D$ we can remove the criterion that the function is nonvanishing on the {\em boundary} of $D$. See  Andersson \cite{Andersson5} and Gauthier \cite{Gauthier}.}.
A related result of Voronin \cite{Voronin1} that predates his universality result is the following:
\begin{thm}
 Suppose  $\halv<\sigma \leq 1.$ Then the set of $n$-tuples 
 \begin{gather*}
   \{ (\zeta(\sigma+it),\zeta'(\sigma+it),\ldots,\zeta^{(n-1)}(\sigma+it)) : t \in \R \} 
\end{gather*}
 is dense in $\C^n$.
\end{thm}
For $1/2<\sigma<1$ this can also be proved quite easily from Theorem 1, which by the power series expansion of an analytic function can be viewed to be an infinite-dimensional version of Theorem 3. Theorem 3 for  $\sigma=\halv$ was recently proved to be false by Garunk{\v{s}}tis and Steuding \cite[Theorem 1 $(ii)$.]{GarSteu}:
\begin{thm}
     The set of  pairs  \begin{gather*} \left \{\p{\zeta\p{\frac 1 2+it},\zeta'\p{\frac 1 2+it}} : t \in \R \right \} \end{gather*}
 is not dense in $\C^2$ 
\end{thm}

They remarked that this implies that the Voronin universality theorem cannot be  extended to any region that covers the line $\Re(s)=1/2$.
 While it is true that it means that Theorem 3 does not extend to the line $\Re(s)=\halv$ it does not follow that Theorem 2 does not extend to the line $\Re(s)=\halv$. 
   Since an interval in the complex plane has an empty interior it does not contain any neighborhood of any point on the critical line.
 In fact it is quite easy to prove that any complex valued function $f(t)$ on an interval can be estimated up to $\epsilon$ in sup-norm by a complex valued polynomial $P(t)$ such that $|P'(t)| \geq  1$.
 By the mean-value theorem this is not true for a real-valued function. 

It is also clear that Theorem 3 on a line does not imply Theorem 2 on a line, an example is the Riemann zeta-function, where  we proved that Theorem 2 is false on $\Re(s)=1$  \cite[Theorem 2]{Andersson4}\footnote{This also follows from \cite[Theorem 7]{Andersson5} and \cite[Remark 4]{Andersson6}.} or, while Voronin proved that Theorem 3 is true.

A natural question to ask is the following: On what lines is the Riemann zeta-function universal? This is  answered in \cite{Andersson4} assuming the Riemann hypothesis. However the the main line considered in that paper is its abscissa of convergence $\Re(s)=1$. In this paper we prove the required result for the critical line $\Re(s)=\halv$. Also we manage to relax the condition of the Riemann hypothesis to that of the Lindel\"of hypothesis.

We prove that Theorem 2 does not extend to the line $\Re(s)=\halv$, by a similar argument as in Garunk{\v{s}}tis-Steuding's paper, by using  the functional equation of the Riemann zeta-function. However, instead of working with the logarithmic derivative we will work with the logarithm of the zeta-function directly. It should be remarked that this proof method of Garunk{\v{s}}tis-Steuding has found other applications in this theory. For example, Kalpokas-Steuding \cite{KalpokasSteuding} proved that the zeta-function takes arbitrarily large real values on the critical line. For these type of results, see also Christ-Kalpokas-Steuding \cite{ChristKalpokasSteuding}.

\section{Non universality on the critical line - first argument}
It has long been known\footnote{I learned of this from Garunk{\v{s}}tis, and according to Steuding it has long been known,  mathematical folklore?} conditionally that Theorem 2 does not extend  to the critical line since the Riemann hypothesis implies that the zeroes will lie to densely. Littlewood \cite{Littlewood} (see \cite[Theorem 14.13]{Titchmarsh}) proved that the Riemann hypothesis implies the estimate
\begin{gather} \label{ui}
    N(T+\delta)-N(T)=\frac{\delta \log  T}{2 \pi} +\Oh{\frac {\log T} {\log \log T} }, \qquad 0<\delta \leq 1.
\end{gather}
Hence there will be a zero in every interval of fixed length for $\zeta(\halv+iT+it)$ provided $T$ is large enough. Thus, these translates of the zeta-function can not 
approximate for example the function $f(t)=1$ on the interval in sup-norm.

There are some problems with this argument:

Firstly, of course we would like to have an unconditional result. In order for the estimate \eqref{ui} to be known to be true we need the Riemann hypothesis. Also we need to prove that some good proportion of these zeros lies on the critical line. This would of course also follow from the Riemann hypothesis, however it is not known unconditionally.

Secondly: While sup-norm is the traditional norm used in universality, this follows from the fact that universality in $L^2$ sense in some neighborhood of $D$ implies universality in sup-norm on $D$. This is a simple consequence of the classical theory of complex functions, for example  Cauchy's theorem.
In fact  in the proofs of universality, Hilbert space arguments are used and the natural norm is the $L^2$-norm. When considering universality on lines and curves, we do not know anything about universality properties in some neighborhood of these lines and curves, and the arguments that shows that $L^2$-norm universality implies sup-norm universality do not apply. Therefore we believe that the natural norm in these cases should be $L^2$-norm.

\section{Main Theorems}

\subsection{Non universality on the critical line}
We will state our two theorems that proves that the Riemann zeta-function is not universal in a strong sense on the critical line:

\begin{thm}
 Suppose $G(t)$ is any real valued function that is locally $L^1$, and that $f \in C(0,H)$. Then 
 \begin{gather*}\liminf_{T \to \infty} \int_0^H \abs{f(t)-G(t+T) \zeta\p{\halv+iT+it}} dt\geq \frac {2} {\pi} \int_0^H |f(t)|dt.\end{gather*} 
\end{thm}
This proves that either $f(t)=\zeta(\halv+iT+it)$ for some $t$, or $f(t)$ can not be approximated by translates of the Riemann zeta-function. In the case when $f(t)=\zeta(\halv+iT+it)$ it can furthermore not be approximated by any other $T' \neq T$. Now it is easy to show that the Riemann zeta-function can not be real valued in any interval on the critical line (we prove this rigorously later in this paper) and by this we immediately obtain:
\begin{thm}
 Let $G$ be any real valued function that is locally $L^1$. If $f \in C(0,H)$  is a real valued function, then
\begin{gather*}
    \inf_T \int_0^H \abs{f(t)-G(T+t) \zeta\p{\halv+it+iT}}dt >0, 
 \end{gather*}
 unless $f$ is identically zero.  
\end{thm}
This shows that no real valued function can be approximated by the translates of the Riemann zeta-function on the critical line.

\subsection{On what lines is the zeta-function universal?}

In our paper \cite{Andersson4} we use this result and a lower bound for the Riemann zeta-function in short intervals to classify the lines where we have universality. Unfortunately we did not manage to prove this result unconditionally, but needed to use the Riemann hypothesis. We will here show how we can relax the condition to assume the Lindel\"of hypothesis instead. We have the following Lemma:
\begin{lem}
  Let $\epsilon>0$ be positive. Assuming the Lindel\"of hypothesis there exist a positive constant $C_\epsilon>0$ such that we have that
 \begin{gather*}
  \int_T^{T+\delta} \abs{\zeta(\sigma+it)} dt \geq (TC_\epsilon)^{-\frac \epsilon \delta}
  \end{gather*}
   for $T \geq 1$, $0<\delta \leq 1$ and $\halv \leq \sigma \leq 1$.
\end{lem}

It is clear that  the lower bound in Lemma 1 can not be improved to a positive constant, since by using Theorem 2 with $f(t)=0$ shows that
\begin{gather}
\liminf_{T \to \infty}\int_T^{T+\delta} \abs{\zeta(\sigma+it)}dt =0, \qquad \qquad \p{\halv<\sigma<1}.
\end{gather}
 The Riemann hypothesis will yield a somewhat sharper lower bound. It would be interesting to understand the true lower bound in Lemma 1.

In view of Lemma 1, we can now proceed as in the proof of Theorem 3 \cite{Andersson4}:
\begin{thm}
 Under the assumption of the Lindel\"of hypothesis we have that the only lines where the Riemann zeta-function is universal in $L^1$-norm and for some interval $[0,\delta]$ are the lines $\Re(s)=\sigma$ for $\halv<\sigma<1$. Furthermore for those lines we have universality for any interval $[0,H]$ and in sup-norm.
\end{thm}

\begin{proof}  For the case $1/2<\sigma<1$ where we have universality, the sup-norm case is exactly Theorem 2. Universality in $L^1$-norm  means that any function $f(t)$ in $L^1(0,H)$ should be possible to be approximated by the Riemann zeta-function $\zeta(\sigma+iT+it)$. Since $C(0,H)$ is dense in $L^1(0,H)$,  Rudin \cite[Theorem 3.14] {Rudin}, this also follows from Theorem 2. 

Regarding when we do not have universality: The only difficult remaining lines  to consider are the  lines $\Re(s)=\sigma$ for $0<\sigma<\halv$. Here we will need the Lindel\"of  hypothesis. From Lemma 1, together with the functional equation and Stirling's formula for the Gamma-factors we obtain that
\begin{gather*} 
 \int_{T}^{T+\delta} \abs{\zeta(\sigma+it)}dt \gg_{\epsilon,\delta}  T^{\halv-\sigma-\epsilon}, \qquad   \qquad \p{0<\sigma<\halv}.
\end{gather*}
Since this will tend to infinity as $T \to \infty$ we see that we do not have universality in $L^1$-norm on the line  $\Re(s)=\sigma$ for $0<\sigma<\halv$. That we do not have universality on the lines $\Re(s)=\sigma$ with $\sigma \leq 0$ and the lines that are not parallel to the imaginary axis follows trivially from the functional equation and the definition of the Riemann zeta-function.
\end{proof}

We would like to restate our problem \cite[Problem 1]{Andersson4}:

\begin{prob}
 Prove Theorem 7 unconditionally.
\end{prob}
While we can not manage this, at least we have weakened the condition from that of the Riemann hypothesis to that of the Lindel\"of hypothesis. In addition,  we will prove two propositions that gives us unconditional, albeit somewhat weaker results in the case $0<\sigma<\halv$.

\section{Proof of Lemma 1}
The main tools of proof for Lemma 1 are Jensen's formula and Jensen's inequality. We first define
\begin{gather} \label{logdef} \log^+ x=\begin{cases} \log x & x \geq 1 \\ 0 & \text{otherwise} \end{cases} \qquad \text{and} \qquad  \log^- x=\begin{cases} -\log x & 0<x<  1 \\ 0 & \text{otherwise} \end{cases} \end{gather}
so that $\log x=\log^+x-\log^- x$.
\begin{lem} 
  We have that for any $\sigma<1$ that 
\begin{multline*}   \int_{-\infty}^\infty \frac{\log^-|\zeta(\sigma+it+iT)|}{1+t^2} dt   \\ \leq  \int_{-\infty}^\infty \frac{\log^+ |\zeta(\sigma+it+iT)|}{1+t^2} dt-\pi \log\left|\zeta(\sigma+1+iT)\right|-\pi \log \left|\frac{\sigma+iT}{2-\sigma+iT} \right|, \end{multline*}
where the inequality can be replaced by an equality  if and only if the Riemann zeta function has no zeroes with $\Re(s) > \sigma$.
\end{lem}
  
\begin{proof} We recall the Jensen's formula 
on the disc $\{z:|z|<r\}$ with $r<1$ for meromorphic functions with a single pole at $z=b_0$ with $|b_0|<r$ and zeroes $a_1,\ldots,a_n$  with $|a_k|<r$
\begin{gather*} \log |f(0)|=\sum_{k=1}^n \log  \left|\frac {a_k} r \right|  -\log \left|\frac{b_0} r \right| +  \frac 1 {2 \pi} \int_0^{2 \pi} \log |f(re^{i\theta})| d \theta. \end{gather*} 
We will take the limit as $r \to 1$ and as long as the limit of the integrals on the right hand side is convergent and converges to the integral with $r=1$ the sum over the zeroes (which might be infinitely many) will also be convergent and the formula 
\begin{gather} \label{iop} \log |f(0)|=\sum_{k} \log  \left|a_k \right|  -\log |b_0| +  \frac 1 {2 \pi} \int_0^{2 \pi} \log |f(e^{i\theta})| d \theta \end{gather} 
holds. We will now use the  Schwartz-Christoffel map $s=\phi(z)$ mapping the unit disc $\{z \in \C: |z| \leq 1\}$ to the half plane $\Re(s) \geq 0$ given by
$$ s=\phi(z)=\frac{1-z}{z+1}, \qquad \text{and} \qquad  z= \phi^{-1}(s)=\frac{1-s}{1+s}.$$ 
With the change of variable $\displaystyle e^{i \theta}=\frac{1-it}{1+it}$, we get $\displaystyle t=\tan \frac \theta 2$
and $\displaystyle d\theta= \frac {2dt} {t^2+1}$. With $f(z)=\zeta(\sigma+\phi^{-1}(s)+iT)$ we get
\begin{multline*}
  \log|\zeta(\sigma+1+iT)| \\ = \sum_{k=1}^\infty \log |a_k|-\log \left|\frac{\sigma+iT}{2-\sigma+iT} \right| + \frac 1 {\pi}\int_{-\infty}^\infty \frac{\log|\zeta(\sigma+it+iT)|}{1+t^2} dt,
\end{multline*}
where  $|a_k|<1$ comes from the zeros of the Riemann zeta-function on the half plane $\Re(s)>\sigma$. Lemma 2 follows from multiplying by $\pi$, by rearranging the terms and from the definition \eqref{logdef}.
 \end{proof}

\noindent {\em Proof of Lemma 1.}
It follows for $0 < \delta \leq 1$ that
\begin{gather} \label{iq} \int_{T}^{T+\delta} \log^{-}|\zeta(\sigma+it)|dt \leq  2 \int_{-\infty}^\infty \frac{\log^{-}|\zeta(\sigma+it+iT)|dt}{1+t^2}.\end{gather} Now we use Lemma 2. The term $-\pi \log|\zeta(1+\sigma+iT)|$ can be estimated from above by $\displaystyle \pi \log \frac{\zeta(3/2)}{\zeta(3)}<3$. The term coming from the pole will be less than $3$ for $T \geq 1$ and by applying the fact that $\displaystyle \log^+ |\zeta(\sigma+it+iT)| \leq \frac{\epsilon} {2 \pi} \log T+\frac{C_\epsilon-6} 2$ for $T \geq 1$, $0\leq t\leq 1$ and $\displaystyle \sigma \geq \frac 1 2$ which follows by the Lindel\"of hypothesis we can estimate the integral on the right hand side of \eqref{iq} by $\log T+(C_\epsilon-6)$ and we obtain the estimate
\begin{gather} \label{yyy} \int_{T}^{T+\delta} \log^{-}|\zeta(\sigma+it)|dt \leq \epsilon \log T+C_\epsilon. \end{gather}
Since $f(x)=\log^- x $ is a convex function we can apply Jensen's inequality 
\begin{gather} \label{yy} 
\log^- \left (\frac 1 \delta \int_{T}^{T+\delta} |\zeta(\sigma+it)|dt \right) \leq  \frac 1 \delta \int_{T}^{T+\delta} \log^{-}|\zeta(\sigma+it)|dt. \end{gather}
Combining the inequalities \eqref{yyy} and \eqref{yy} and taking the exponential we obtain Lemma 1. \qedsymbol

\section{Non universality for $\Re(s)=\sigma$ with $0<\sigma<\halv$}

By the same method, as in the proof of Lemma 1 it follows unconditionally that given $\epsilon>0$ and $\sigma>\halv$, then there exist some absolute constant $H>0$ such that
\begin{gather*}
 \int_{T}^{T+H} \abs{\zeta(\sigma+it)} dt \gg_\epsilon  T^{-\epsilon}.
\end{gather*}
See Theorem 1.1.1.  and the remark at the end of page 7  of Ramachandra \cite{Ramachandra}. From this we obtain the following proposition
\begin{prop}
 Let $0<\sigma<\halv$. Then there exists a $H>0$ such that the zeta-function is not universal on the line $\Re(s)=\sigma$ for the interval $[0,H]$ in $L^1$-norm. More precisely,   for any continuous function $f(t)$ on $[0,H]$ we have that
\begin{gather*}
 \liminf_{T \to \infty} \int_0^H \abs{f(t)-\zeta(\sigma+it+iT)}dt = \infty.
\end{gather*}
\end{prop}
 Infact by Eq \eqref{ajaj} the constant $H$ in Proposition 1 depending on $\sigma$ can be made explicit (a variant of Ramachandra's result with an explicit $A$ should be used). In order to obtain optimal estimates for $H$,    a more careful choice of $\sigma_1$ and $\sigma_2$ are also needed. For example $\sigma_1=1$ and $\sigma_2=2-\sigma$ will give better constants. Then however we would need to use the lower bounds of \cite[Theorem 2]{Andersson4} instead of the trivial bounds  to estimate $J(\sigma_1)$.

There is a special feature of the classical universality theorems, namely that of ``positive lower density''. More precisely in Theorem 1 it is proved that
\begin{gather*}
 \liminf_{T \to \infty} \frac 1 T \operatorname{meas} \{\max_{z \in D} |f(z)-\zeta(z+iT)|<\epsilon \} >0.
\end{gather*}
The same holds true for Theorem 2.
 These results follows from the proofs of universality but should in our opinion not be regarded as an essential part of the definition of universality.  In the case of lines $\Re(s)=\sigma$ for $0<\sigma<\halv$   it is in fact much easier to prove that we do not have positive density universality, than regular universality. In fact from classical zero-density estimates, for example, from the following  result of Ingham \cite{Ingham} (see also \cite[Chapter 11]{Ivic}):
\begin{gather}
  N(\sigma,T)\ll  T^{3(1-\sigma)/(2-\sigma)}\log^5 T
 \end{gather}
it follows that
\begin{prop}
 The zeta-function is not ``positive density'' universal on the line $\Re(s)=\sigma$ for $0<\sigma<\halv$. More precisely: Let $0<\sigma<\halv$ and $H>0$ be fixed  and
 let $f$ be any continuous function on the interval  $[0,H]$. Then 
 \begin{gather*}
  \frac 1 T \operatorname{meas} \{\max_{t \in [0,H]} |f(t)-\zeta(\sigma+it+iT)|<\epsilon \} \ll_{\epsilon,\varepsilon}  T^{(2\sigma-1)/(1+\sigma) + \varepsilon}.
 \end{gather*} 
\end{prop}
It follows from the method of proof that Proposition 2 can  be sharpened by replacing $<\epsilon$ with $<\epsilon T^{\halv-\sigma-\varepsilon}$.

\section{Hardy's $Z$-function}

The Hardy $Z$-function\footnote{This function has been extensively studied. Recent papers includes for example Jutila \cite{Jutila}, Ivic \cite{Ivic2,Ivic3} and Korolev \cite{Korolev}.}  is defined by
\begin{gather}  \label{uiq2}
 Z(t)=e^{i \theta(t)} \zeta \p{\halv+it},
\end{gather}
where $\theta(0)=0$ and $\theta(t)$ is a continuous real valued  function, that is chosen so that $Z(t)$ is a real valued function.

 By the functional equation of the Riemann zeta function we have that (See e.g. Titchmarsh \cite[Section 10.2]{Titchmarsh})
\begin{gather*}
  e^{i \theta(t)} = 
-\halv \p{\frac 1 4 +t^2} \pi^{-1/4-i t/2} \Gamma\p{\frac 1 4 +\frac{it} 2}\zeta\p{\halv+it}. \\ \intertext{Hence}
\theta(t)= \arg \p{- \pi^{- i t/2} \Gamma\p{\frac 1 4 +\frac{it} 2}\zeta\p{\halv+it}}.
\end{gather*}
By the fact that the Hardy $Z$-function is real valued the  argument of $Z(t)$ must therefore be $n \times \pi$ for some integer $n$.
From this we see that
\begin{gather*}
 \arg \zeta \p{\halv+it}=n\times \pi-\theta(t),
\end{gather*}
for each $t$  and some integer $n$ depending on $t$. It is clear that
\begin{gather*}
 |A-B| = \abs{A} \cdot \abs{1-A/B} \geq |A| \cdot \abs{\Im (1-A/B)} = |A| \cdot |\sin \arg(A/B)|
\end{gather*}
Hence assuming that $f(t)=e^{i \eta(t)} |f(t)|$ we obtain that
\begin{gather} \label{eq9} 
   \int_0^H \abs{f(t)-\zeta \p{\halv+it+iT}} dt  \geq  \int_0^H \abs{f(t)} \cdot \abs{\sin (\theta(t+T)-\eta(t))} dt
\end{gather}
From this equation we will obtain a proof of our theorems.
We first prove the following Lemma:
\begin{lem}
 We have that
 \begin{gather*} \lim_{T \to \infty} \int_{A}^B |\sin (\theta(t+T)+C)| dt =\frac{2(B-A)}{\pi}.\end{gather*}
\end{lem}

\begin{proof}
 By Stirling's formula we have that \begin{gather*} \theta(t) =  t \log t+a t + b +O(1/t).\end{gather*}
 Hence 
\begin{gather*}\theta(t+T)=  t \log T + D  + O(1/T).\end{gather*} By the substitution $x=t \log T+D+C$  we see that
 \begin{align*} \int_A^B \abs{\sin (\theta(t+T)+C)}dt &= \int_{A}^B  |\sin \theta(t \log T+C+D)| dt + \Oh{\frac {|A-B|}T} \\
&=    \frac 1 {\log T} \int_{A  \log T-D-C}^{B  \log T-D-C} |\sin  x| dx  + \Oh{\frac {|A-B|}T} \\ &= \frac{2(B-A)}{\pi}+ \Oh{\frac 1 {\log T}} + \Oh{\frac {|A-B|}T},
 \end{align*}
by the fact that $\abs{\sin x} $ has period $\pi$ and
\begin{gather*}
  \int_0^{\pi} \abs{\sin x}  dx= 2 \int_0^{\pi/2} \sin x \, dx =2.
\end{gather*}
\end{proof}

\section{Proof of Theorem 5}
 The proof essentially follows by estimating the integral with respect to the function $f$ into Riemann sums. By dividing the integral
\begin{gather*}
 (*)=  \int_0^H \abs{f(t)-G(T+t)\zeta \p{\halv+it+iT}}  dt, \\ \intertext{into $N$ sub-intervals, we obtain}
(*)=\sum_{k=1}^N  \int_{(k-1)/N}^{kH/N} \abs{f(t)-G(T+t)\zeta \p{\halv+it+iT}}  dt.
\end{gather*}
Since $f(t)$ is continuous on the compact interval $[0,H]$, we can choose $N$ so that
\begin{gather} \label{eqab}
|f(x)-f(x_0)|<\epsilon,  \qquad \text{whenever} \qquad |x-x_0|<1/N. 
\end{gather}
 We obtain

\begin{gather*} (*) = \sum_{k=1}^N  \int_{(k-1)/N}^{kH/N} \abs{f(kH/N)-G(T+t)\zeta \p{\halv+it+iT}}dt + \Oh{H \epsilon}.  \end{gather*}
Let us define $C_k$ such that  $f(kH/N)= e^{iC_k} |f(kH/N)|$. By \eqref{eq9} we get
\begin{gather*} (*)   \geq \sum_{k=1}^N  f(kH/N) \int_{(k-1)/N}^{kH/N} |\sin (\theta(t+T)+C_k)|dt + \Oh{H \epsilon}. \end{gather*}
 By using Lemma 3 on the integrals on the right hand side we see that 
\begin{gather*} \liminf_{T \to \infty} \int_0^H  \abs{f(t)-G(T+t)\zeta \p{\halv+it+iT}}  dt \geq  \frac {2H} {\pi N} \sum_{k=1}^N f \p{\frac {kH} N} + \Oh{H \epsilon}. \end{gather*}
By estimating the Riemann sum on the right by its integral and using \eqref{eqab} we obtain that
\begin{gather*} \liminf_{T \to \infty} \int_0^H  \abs{f(t)-G(T+t)\zeta \p{\halv+it+iT}}  dt \geq  \frac 2 \pi \int_0^H |f(t)| dt  + \Oh{H \epsilon}. \end{gather*}
Our result follows by letting $\epsilon \to 0$. \qed

 \section{Proof of Theorem 6}
By Theorem 5 we have 
\begin{gather*}
 \liminf_{T \to \infty}  \int_0^H \abs{f(t)-\zeta \p{\halv+it+iT}} dt  \geq   \frac 2 {\pi} \int_0^H \abs{f(t)}  dt >0. 
\end{gather*}
In particular it means that for any non-zero function there exist a $T_0$ and an $\epsilon>0$ such that if $|T| \geq T_0$ then
\begin{gather*}
 \int_0^H \abs{f(t)-\zeta \p{\halv+it+iT}} dt  \geq  \epsilon.
\end{gather*}
Thus it is sufficient to consider $|T| \leq T_0$. Let us now choose $|T| \leq T_0$ that minimize the integral.
\begin{gather} \label{uiq}
 \int_0^H \abs{f(t)-\zeta \p{\halv+it+iT}} dt.  
\end{gather}
 Since $f(t)$ is a real valued continuous function we have that
 the argument of $f(t)$ is constant in some small interval. However since $\theta(t)$ defined by \eqref{uiq2} is an analytic non-constant function it can not be constant in any interval. This means that the function $\zeta(\halv+iT+it)$ can not equal the function $f(t)$ on any small interval, and we have that the integral \eqref{uiq} is strictly positive. \qed

\section{Possible universality on the critical line}
\subsection{Universality of the Hardy $Z$-function?}

Once we know that the Riemann zeta-function is not universal on the critical line, and the reason is that the argument of the Riemann zeta-function is to regular, we might ask what happens when we remove this regularity. A natural question  is therefore to ask whether the Hardy $Z$-function admits some universality property:
\begin{prob}
 Is the Hardy $Z$-function  universal in $L^1$-norm? If $f \in L^1(0,H)$ is a real valued  function, is it true that 
  \begin{gather*} \liminf_{T \to \infty} \int_0^H |Z(T+t)-f(t)|dt=0? \end{gather*}
\end{prob}
The problem can be asked with $L^1$ norm replaced by $L^2$-norm, and with some  weight on the Hardy function, for example with $Z(t+T)$ replaced by $Z(t+T)/\sqrt{\log (t+T)}$. However, from the same arguments that we used in Section 2, that the zeroes of the Riemann zeta-functions are ``too dense'' it follows that no nonzero function can be approximated by the Hardy $Z$-function in sup-norm. Thus we know that we do not have universality (at least if the Riemann hypothesis is assumed) in sup-norm.  

It might be easier to solve the problem if the absolute value of the Hardy Z-function, or equivalently the absolute value of the Riemann zeta-function is considered:
\begin{prob}
  Is the absolute value of the Riemann zeta-function universal on the critical line in $L^1$-norm? If $f \in L^1(0,H)$ is a positive function is it true that
  \begin{gather*} \liminf_{T \to \infty} \int_0^H \abs{\abs{\zeta\p{\halv+iT+it}}-f(t)}dt=0?\end{gather*}
\end{prob}
Similarly to the case of the Hardy $Z$-function proper we can consider weighted version of universality and other norms (although not sup-norm).

\subsection{Universality of related functions}

 Selberg (unpublished) proved that $\abs{\zeta \p{\halv+it}}^{1/\sqrt{\log \log t}}$ admits a limit distribution. For the first published proof, see Joyner \cite{Joyner}. Hejhal \cite{Hejhal} proved a somewhat sharper and more explicit variant of this statement. 

 Related work has been done by Laurin{\v{c}}ikas (\cite{Laurincikas2}, \cite{Laurincikas3} and  \cite[Chapter 7]{Laurincikas}). For example he disproved a conjecture of Bagchi \cite{Bagchi} that  $\zeta(\halv+it)$ admits a limit distribution in the space of continuous functions, and conjectured that $\zeta(\halv+it)^{1/\sqrt{\log \log t}}$ does admit this limit distribution. 

These results suggest to us that on the critical line the natural function to consider might be
\begin{gather} \label{ppi}
 \zeta\p{\halv+it}^{1/\sqrt{\log \log t}}
\end{gather}
instead of $\zeta(\halv+it)$ proper. We remark here that this function is not uniquely defined and must be defined through analytic continuation through a path. Natural paths $\gamma(t)$ might be the ones that can be parametrized by $ \gamma(t)=t+i(\halv-\sigma(t))$, such that $\halv+i \gamma(t)$  stays   on the right hand side of the zeros (and the pole) of the zeta-function in the complex plane. This class of paths will give  a precise definition
of the function in \eqref{ppi}, except  for values of $t$  where $\zeta(\sigma+it)=0$ and $\sigma>\halv$ (zeroes violating the Riemann hypothesis) where the function might have a discontinuity.
 We therefore ask the following question regarding universality:
\begin{prob}
  Is it true for any complex valued  continuous function $f$ on the interval $[0,H]$  that
  \begin{gather*} \liminf_{T \to \infty} \int_0^H \abs {\zeta\p{\halv+iT+it}^{1/\sqrt{\log \log (t+T)}}-f(t)}dt=0?
 \end{gather*}
\end{prob}
Here the function must be defined  precisely defined through analytic continuation through a path as discussed above. A way to simplify the problem and avoid the problem of the function not beeing uniqely defined is to consider the related problem with absolute values:
\begin{prob}
   Is it true that for any  positive continuous function on the interval $[0,H]$ that
  \begin{gather*}\liminf_{T \to \infty} \int_0^H \abs{\abs{ \zeta \p{\halv+iT+it}}^{1/\sqrt{\log \log (t+T)}}-f(t)}dt=0?
\end{gather*}
\end{prob}

As before it is true that questions in Problems 4 and Problems 5 can not be true (at least assuming RH) in sup-norm, by the ``density of zeros'' argument.
\begin{ack} The author is grateful to Aleksandar Ivi\'c who noticed serious mistakes in the proof of Lemma 1 in v1 of this paper. The proof is now rewritten using a new method. \end{ack}
\bibliographystyle{plain}

\def\cprime{$'$} \def\polhk#1{\setbox0=\hbox{#1}{\ooalign{\hidewidth
  \lower1.5ex\hbox{`}\hidewidth\crcr\unhbox0}}}
  \def\polhk#1{\setbox0=\hbox{#1}{\ooalign{\hidewidth
  \lower1.5ex\hbox{`}\hidewidth\crcr\unhbox0}}}

\end{document}